\documentclass[10pt,a4paper]{article}
\usepackage[utf8]{inputenc}
\usepackage{amsmath}
\usepackage{amsfonts}
\usepackage{amssymb}
\author{Andreas Basse-O'Connor, Jan Pedersen,
Mikkel Slot Nielsen, \\ and Victor Rohde \\
Department of Mathematics \\
Aarhus University \\
\{basse, jan, mikkel, victor\}@math.au.dk
}

\usepackage{amssymb}
\usepackage{amsthm}
\usepackage{amscd}
\usepackage{amstext}
\usepackage{fancyhdr}
\usepackage{amsthm}
\usepackage[numbers]{natbib}
\usepackage[T1]{fontenc}
\usepackage{enumerate}
\usepackage[american]{babel}
\usepackage{graphicx}
\usepackage{bbm}
\usepackage[mathscr]{euscript}
\usepackage{mathrsfs}
\usepackage{stmaryrd}
\usepackage{soul}
\usepackage{hyperref}
\usepackage{xspace}
\usepackage[draft,danish]{fixme}
\usepackage{mathtools}
\usepackage{dsfont}
\usepackage{url}
\usepackage[ps,all,dvips]{xy}
\SelectTips{cm}{12}
\CompileMatrices

  \makeatletter
 \useshorthands{"}
 \defineshorthand{"-}{\nobreak-\bbl@allowhyphens}
 \makeatother

\newcommand{\EE}{\mathbb{E}} 

\newcommand{\NN}{\ensuremath{\mathbb{N}}}           % The natural numbers
\newcommand{\RR}{\ensuremath{\mathbb{R}}}           % The real numbers

\newcommand\eqindist{\stackrel{d}{=}}
\newcommand{\trace}{\operatorname{trace}} 
\newcommand{\diag}{\operatorname{diag}} 

\theoremstyle{plain}
\newtheorem{lemma}{Lemma}[section]

\newtheorem{theorem}[lemma]{Theorem}
\newtheorem{corollary}[lemma]{Corollary}
\theoremstyle{definition}
\newtheorem{example}[lemma]{Example}

\theoremstyle{definition}
\newtheorem{definition}[lemma]{Definition}
\theoremstyle{remark}

\newtheorem{remark}[lemma]{Remark}

\newcommand{\devnull}[1]{}

\numberwithin{equation}{section}

\date{}
\begin{document}

%\begin{frontmatter}

\title{On infinite divisibility of a class of two-dimensional vectors in the second Wiener chaos}

%% Group authors per affiliation:
\author{Andreas Basse-O'Connor, Jan Pedersen, and Victor Rohde}

%\address{Department of Mathematics, Aarhus University, Ny Munkegade 118, 8000 Aarhus C, Denmark. E-mail adresses: basse@math.au.dk, jan@math.au.dk, victor@math.au.dk}

%\address[mymainaddress]{1600 John F Kennedy Boulevard, Philadelphia}
%\address[mysecondaryaddress]{360 Park Avenue South, New York}

\numberwithin{equation}{section}

\maketitle

\begin{center} \footnotesize Department of Mathematics, Aarhus University, \textit{E-mail adresses:} basse@math.au.dk (A.~Basse-O'Connor), jan@math.au.dk (J. Pedersen), victor@math.au.dk (V. Rohde) \end{center}

\normalsize
\begin{abstract}

Infinite divisibility of a class of two-dimensional vectors with components in the second Wiener chaos is studied. Necessary and sufficient conditions for infinite divisibility is presented as well as more easily verifiable sufficient conditions. The case where both components consist of a sum of two Gaussian squares is treated in more depth, and it is conjectured that such vectors are infinitely divisible.  

\end{abstract}

%\begin{keyword}
%\end{keyword}
\normalsize
%\end{frontmatter}

%\linenumbers

\section{Introduction}

Paul Lévy \cite{LEVY} raised the question of infinite divisibility of Gaussian squares, that is, for a centered Gaussian vector $(X_1,\dots,X_n)$ when can $(X_1^2,\dots,X_n^2)$ be written as a sum of $m$ independent identical distributed random vectors for any $m\in \NN$? Several authors have studied this problem. We refer to \cite{EISENBAUM2016,EISENBAUMKASPI,EVANS,GRIFFITHS1970,GRIFFITHS,VEREJONES} and reference therein. These works include several novel approaches and gives a great understanding of when Gaussian squares are infinitely divisible. In this paper we will provide a characterization of infinite divisibility of sums of Gaussian squares which to the best of our knowledge has not been studied in the literature except in special cases. This problem is highly motivated by the fact that sums of Gaussian squares are the usual limits in many limit theorems in the presence of either long range dependence, see \cite{DOBRUSHIN} or \cite{TAQQU}, or degenerate U-statistics, see \cite{JANSON}. In the following we will go in more details.  

Let $Y$ be random variable in the second (Gaussian) Wiener chaos, that is, the closed linear span in $L^2$ of $\{W(h)^2-1 : h \in H, \Vert h \Vert =1 \}$ for a real separable Hilbert space $H$ and an isonormal Gaussian process $W$. For convenience, we assume $H$ is infinite-dimensional. Then there exists a sequence of independent standard Gaussian variables $(\xi_i)$ and a sequence of real numbers $(\alpha_i)$ such that 
\begin{align*}
Y \eqindist \sum_{i=1}^\infty \alpha_i (\xi_i^2 -1),
\end{align*}
where the sum converges in $L^2$ (see for example \cite[Theorem 6.1]{JANSON}). Since the $\xi_i$'s are independent, $(\xi_1^2,\dots,\xi^2_d)$ is infinitely divisible for any $d \geq 1$ and therefore, $Y$ is infinitely divisible. Such a sum of Gaussian squares appears as the limit of U-statistics in the degenerate case (see \cite[Corollary 11.5]{JANSON}). In this case the $\alpha_i$ are certain binomial coefficients times the eigenvalues of operators associated to the U-statistics. We note that the sequence $(\xi_i)$ depends heavily on $Y$, so one can not deduce joint infinite divisibility of random vectors with components in the second Wiener chaos. In particular, for a vector with dimension greater than or equal to three and components in the second Wiener chaos it is well known (cf.\ Theorem \ref{MarcusAndRosen} below) that it need not be infinite divisibility. In between these two cases is the open question of infinite divisibility of a two-dimensional vector with components in the second Wiener chaos. Let $(X_1,\dots, X_{n_1+n_2})$ be a mean zero Gaussian vector for $n_1,n_2 \in \NN$. That any two-dimensional vector in the second Wiener chaos is infinitely divisible is equivalent to
\begin{align}\label{Question}
(d_1 X_1^2+\dots + d_{n_1} X_{n_1}^2,d_{n_1}X_{n_1 +1 }^2+\dots +d_{n_1+n_2}X_{n_1+ n_2}^2)
\end{align}
being infinitely divisible for any $d_1,\dots,d_{n_1+n_2} = \pm 1$, any covariance structure of $(X_1,\dots,X_{n_1+n_2})$, and any $n_1,n_2 \in \NN$ (something that follows by the definition of the second Wiener chaos).

The following theorem, which is due to Griffiths \cite{GRIFFITHS} and Bapat \cite{BAPAT}, is an important first result related to infinite divisibility in the second Wiener chaos. We refer to Marcus and Rosen \cite[Theorem 13.2.1 and Lemma 14.9.4]{MARCUS} for a proof.

\begin{theorem}[Griffiths and Bapat]\label{MarcusAndRosen}
Let $(X_1,\dots,X_n)$ be a mean zero Gaussian vector with positive definite covariance matrix $\Sigma$. Then $(X_1^2,\dots,X_n^2)$ is infinitely divisible if and only if there exists an $n\times n$ matrix $U$ on the form $ \diag(\pm 1,\dots,\pm 1)$ such that $U^t \Sigma^{-1} U$ has non-positive off-diagonal elements.
\end{theorem}

This theorem resolved the question of infinite divisibility of Gaussian squares.  For $n \geq 3$ there is an $n \times n$ positive definite matrix $\Sigma$ where there does not exist an $n \times n$ matrix $U$ on the form $\diag(\pm 1,\dots ,\pm 1)$ such that $U^t \Sigma^{-1} U$ has non-positive off-diagonal elements. Consequently, there are mean zero Gaussian vectors $(X_1,\dots, X_n)$ such that $(X_1^2,\dots ,X_n^2)$ is not infinite divisible whenever $n \geq 3$. 

Eisenbaum \cite{EISENBAUM} and Eisenbaum and Kapsi \cite{EISENBAUMKASPI} found a connection between the condition of Griffiths and Bapat and the Green function of a Markov process. In particular, a Gaussian process has infinite divisible squares if and only if its covariance function (up to a constant function) can be associated with the Green function of a strongly symmetric transient Borel right Markov process.

When discussing the infinite divisibility of the Wishart distribution Shanbhag \cite{SHANBHAG} showed that for any covariance structure of a mean zero Gaussian vector $(X_1,\dots,X_n)$, 
\begin{align*}
(X_1^2,X_2^2+ \dots + X_n^2)
\end{align*}
is infinitely divisible. Furthermore, it was found that infinite divisibility of any bivariate marginals of a centered Wishart distribution can be reduced to infinite divisibility of $(X_1X_2,X_3X_4)$. By the polarization identity, 
\begin{align*}
(X_1X_2,X_3X_4) = \tfrac{1}{4} ( (X_1 +X_2)^2 - (X_1 - X_2)^2, (X_3 + X_4)^2 - (X_3 - X_4)^2). 
\end{align*} 
Consequently, infinite divisibility of any bivariate marginals of a centered Wishart distribution is again related to the question of infinite divisibility of a two-dimensional vector from the second Wiener chaos.

We will be interested in the infinite divisibility of 
\begin{align*}
(X_1^2+\dots +X_{n_1}^2,X_{n_1 +1}^2 +\dots +X_{n_1+n_2}^2),
\end{align*}
i.e., the case $d_1=\dots=d_{n_1+n_2}=1$ in (\ref{Question}). The general case, where $d_i=-1$ for at least one $i$, seems to require new ideas going beyond the present paper. We will have a special interest in the case $n_1 = n_2 = 2$. 

Despite the simplicity of the question, it has proven rather subtle, and a definite answer is not presented. Instead, we give easily verifiable conditions for infinite divisible in the case $n_1 = n_2 =2$ as well as more complicated necessary and sufficient conditions in the general case that may or may not always hold. We will, in addition, investigate the infinite divisibility of $(X_1^2+X_2^2,X_3^2 +X_4^2)$ numerically which, together with Theorem \ref{Suffcond} (ii), leads us to conjecture that infinite divisibility of this vector always holds.

The main results without proofs are presented in Section 2. Section 3 contains two examples and a small numerical discussion. We end with Section 4 where the proofs of the results stated in Section 2 are given.

\section{Main Results}

We begin with a definition which is a natural extension to the present setup (see the proof of Corollary \ref{infcor1}) of the terminology used by Bapat \cite{BAPAT}. 

\begin{definition}
Let $n_1,n_2 \in \NN$. An $(n_1+n_2) \times (n_1+n_2)$ orthogonal matrix $U$ is said to be an $(n_1,n_2)$-signature matrix if 
\begin{align*}
U = \begin{pmatrix}
U_1 & 0 \\ 
0 & U_2  \\
\end{pmatrix}
\end{align*}
where $U_1$ is an $n_1\times n_1$ matrix and $U_2$ is an $n_2\times n_2$ matrix, both orthogonal, and for $0$'s of suitable dimensions. 
\end{definition}

Let $n_1,n_2 \in \NN$ and consider a mean zero Gaussian vector $(X_1,\dots,X_{n_1+n_2})$ with positive definite covariance matrix $\Sigma$. Now we present a necessary and sufficient condition for infinite divisibility of 
\begin{align}\label{genvec}
(X_1^2+\dots +X_{n_1}^2,X_{n_1 +1}^2 +\dots +X_{n_1+n_2}^2).
\end{align}
For $a > 0$, let $Q = I - (I+ a \Sigma)^{-1}$ and write 
\begin{align*}
Q = \begin{pmatrix}
Q_{11} & Q_{12} \\
Q_{21} & Q_{22}
\end{pmatrix}
\end{align*}  
where $Q_{11}$ is an $n_1 \times n_1$ matrix, $Q_{22}$ is an $n_2 \times n_2$ matrix, and $Q_{12}= Q_{21}^t$ (where $Q_{21}^t$ is the transpose of $Q_{21}$) is an $n_1 \times n_2$ matrix. Note that if $\lambda$ is an eigenvalue of $\Sigma$, $\frac{a \lambda}{1+a\lambda}$ is an eigenvalue of $Q$. Since $Q$ is symmetric and has positive eigenvalues, it is positive definite.

\begin{theorem}\label{Tracethm}

The vector in (\ref{genvec}) is infinitely divisible if and only if for all $k,m \in \NN_0$ and for all $a>0$ sufficiently large,
\begin{align}\label{sumoftrace}
\begin{aligned}
\MoveEqLeft \sum \: \trace\: Q_{11}^{k_1} Q_{12}Q_{22}^{m_1} Q_{21} Q_{11}^{k_2} \cdots Q_{11}^{k_{d} }Q_{12}Q_{22}^{m_{d}} Q_{21} Q_{11}^{k_{d+1}}  \\
 &+ \sum \: \trace\: Q_{22}^{m_1} Q_{21}Q_{11}^{k_1} Q_{12} Q_{22}^{m_2}  \cdots Q_{22}^{m_{d-1} }Q_{21}Q_{11}^{k_{d}} Q_{12} Q_{22}^{m_{d+1}} \geq 0,
 \end{aligned}
\end{align}
where the first sum is over all $k_1,\dots, k_{d+1}$ and $m_1, \dots, m_d$ such that 
\begin{align*}
k_1 + \dots + k_{d+1} + d  = k \quad \text{and} \quad m_1 + \dots + m_{d} + d = m,
\end{align*}
and the second sum is over all $m_1,\dots, m_{d+1}$ and $k_1, \dots, k_d$ such that 
\begin{align*}
m_1 + \dots + m_{d+1} + d = m \quad \text{and} \quad  k_1 + \dots + k_{d} + d = k.
\end{align*}
\end{theorem}

\begin{remark}
By applying Theorem \ref{Tracethm} we can give a new and simple proof of Shanbhag's \cite{SHANBHAG} result that $(X_1^2,X_2^2+\dots + X_{1+n_2}^2)$ is infinite divisible. To see this, consider the case $n_1=1$ and $n_2 \in \NN$. Then $Q_{11}$ is a positive number and $Q_{12} Q_{22}^m Q_{21}$ is a non-negative number for any $m \in \NN$. In particular, we have 
\begin{align*}
\MoveEqLeft \trace Q_{11}^{k_1} Q_{12}Q_{22}^{m_1} Q_{21} \cdots Q_{12}Q_{22}^{m_{d}} Q_{21} Q_{11}^{k_{d+1}} \\
&= Q_{11}^{k_1} \cdots Q_{11}^{k_{d+1}} Q_{12}Q_{22}^{m_1} Q_{21} \cdots Q_{12} Q_{22}^{m_d} Q_{21} \geq 0
\end{align*}
for any $k_1,\dots, k_{d+1}, m_1, \dots , m_d \in \NN_0$. Consequently, the first sum in (\ref{sumoftrace}) is a sum of non-negative numbers. A similar argument gives that the other sum is non-negative too. We conclude that $(X_1^2,X_2^2+\dots + X_{1+n_2}^2)$ is infinite divisible. % It is interesting how short the argument for infinite divisibility is with this approach.
\end{remark}

In order to get a concise formulation of the following results we will need some terminology and conventions. To this end, consider a $2\times 2$ symmetric matrix $A$. Let $v_1$ and $v_2$ be the eigenvectors of $A$, and $\lambda_1$ and $\lambda_2$ be the corresponding eigenvalues. We say that $v_i$ is associated with the largest eigenvalue if $\lambda_i \geq \lambda_j$ for $j=1,2$. Furthermore, whenever $A $ is a multiple of the identity matrix, we fix $(1,0)$ to be the eigenvector associated with the largest eigenvalue.

Now consider the special case $n_1=n_2=2$, i.e., the vector 
\begin{align}\label{whatwewanttoinv2}
(X_1^2+X_2^2,X_3^2+X_4^2)
\end{align}
where $(X_1,X_2,X_3,X_4)$ is a mean zero Gaussian vector with a $4 \times 4$ positive definite covariance matrix $\Sigma$. We still let $Q = I - (I+a \Sigma)^{-1}$ and write 
\begin{align*}
Q = \begin{pmatrix}
Q_{11} & Q_{12} \\
Q_{21} & Q_{22}
\end{pmatrix}
\end{align*}
where $Q_{ij}$ is a $2 \times 2$ matrix for $i,j=1,2$. Let $W$ be a $(2,2)$-signature matrix such that
\begin{align*}
W^tQW = \begin{pmatrix}
W_1^t Q_{11} W_1 & W_1^t Q_{12}W_2 \\
W_2^tQ_{21}W_1 &W_2^t Q_{22}W_2 
\end{pmatrix} =  \begin{pmatrix}
q_{11} & 0 & q_{13} & q_{14} \\
0 & q_{22} &q_{23} &q_{24} \\
q_{13} & q_{23} &q_{33} & 0 \\
q_{14} &q_{24} & 0 & q_{44} 
\end{pmatrix},
\end{align*}
where $q_{11} \geq q_{22}>0$ and $q_{33} \geq q_{44}>0$ which exists by Lemma \ref{existssignmat}. Note that $q_{ij}$ is not the $(i,j)$-th entry of $Q$ but of $W^t Q W$. Let $v_1 = (v_{11},v_{21})$ be the eigenvector of $W_1^t Q_{12}Q_{21} W_1 $ associated with the largest eigenvalue. If $q_{11}=q_{22}$ or $q_{33}=q_{44}$, any orthogonal $W_1$ or $W_2$ gives the desired form. In this case, we may always choose $W_1$ or $W_2$ such that $v_{11} q_{13} ( v_{11} q_{13}+ v_{21} q_{23}) \geq 0$ (see the proof of Lemma \ref{Rotprop}, (ii) $\Rightarrow$ (iii)), and it is such a choice we fix. The following theorem addresses the non-negativity of the sums in (\ref{sumoftrace}) when $n_1=n_2=2$,.

\begin{theorem}\label{Suffcond}
Let $n_1=n_2=2$. Then, in the notation above, we have the following.
\begin{enumerate}
\item[(i)] For all $d \in \NN_0$ and $k_1,\dots,k_{d+1},m_1,\dots,m_d \in \NN_0$,
\begin{align*}
\trace Q_{11}^{k_1} Q_{12}Q_{22}^{m_1} Q_{21} Q_{11}^{k_1} \cdots Q_{11}^{k_{d} }Q_{12}Q_{22}^{m_{d}} Q_{21} Q_{11}^{k_{d+1}} \geq 0
\end{align*} 
if and only if $v_{11} q_{13} ( v_{11} q_{13}+ v_{21} q_{23}) \geq 0$. In particular, (\ref{whatwewanttoinv2}) is infinitely divisible if the latter inequality is satisfied for all sufficiently large $a$. 
\item[(ii)] For any $k,m \in \NN_0$ such that at least one of the following inequalities is satisfied: (i) $k \leq 2$, (ii) $m\leq 2$, or (iii) $k+m \leq 7$, the sum in (\ref{sumoftrace}) is non-negative. 
\end{enumerate}
\end{theorem}

\begin{remark}
When $v_{11} q_{13} ( v_{11} q_{13}+ v_{21} q_{23}) < 0$, we know that there are $k, m \in \NN_0$ such that (\ref{sumoftrace}) with $n_1=n_2=2$ contains negative terms cf.\ Theorem~2.4~(i). If $k=0$ or $m=0$ then Theorem~\ref{Suffcond}~(ii) gives that the sum in (\ref{sumoftrace}) is non-negative. If $k, m \geq 1$, the sum in (\ref{sumoftrace}) always contains terms on the form
\begin{align}\label{postrace}
\trace Q_{11}^{k_1} Q_{12} Q_{22}^{m_1} Q_{21}.
\end{align}
Since $Q_{11}$ is positive definite and $\trace AB = \trace BA$ for any matrices $A$ and $B$ such that both sides make sense,
\begin{align*}
\trace  Q_{11}^{k_1} Q_{12} Q_{22}^{m_1} Q_{21} = \trace  Q_{11}^{k_1/2} Q_{12} Q_{22}^{m_1} Q_{21} Q_{11}^{k_1/2}.
\end{align*}
Using $Q_{12} =Q_{21}^t$ we conclude that (\ref{postrace}) is equal to the trace of a positive semi-definite matrix and therefore non-negative. Consequently, there are always non-negative terms in (\ref{sumoftrace}). 

It is an open problem if there exists a positive definite matrix $Q$ with eigenvalues less than $1$ and $k,m \in \NN_0$ such that (\ref{sumoftrace}) is negative, which would be an example of (\ref{whatwewanttoinv2}) not being infinite divisible, or if the non-negative terms always compensate for possible negative terms, which is equivalent to (\ref{whatwewanttoinv2}) always being infinitely divisible. 
\end{remark}

Continue to consider the case $n_1=n_2=2$ and write 
\begin{align*}
\Sigma^{-1} = \begin{pmatrix}
\Sigma^{11} & \Sigma^{12} \\
\Sigma^{21} & \Sigma^{22}
\end{pmatrix}
\end{align*} 
where $\Sigma^{ij}$ is a $2\times 2$ matrix for $i,j=1,2$. %In Theorem \ref{infcor} we give conditions on $\Sigma^{-1}$ ensuring infinite divisibility of (\ref{whatwewanttoinv2}). 
Let $W$ be a $(2,2)$-signature matrix such that 
\begin{align*}
W^t \Sigma^{-1} W = \begin{pmatrix}
W_1^t\Sigma^{11} W_1 & W_1^t \Sigma^{12} W_2 \\
W_2^t \Sigma^{21}W_1 & W_2^t \Sigma^{22} W_2 
\end{pmatrix} =  \begin{pmatrix}
\sigma_{11} & 0 & \sigma_{13} & \sigma_{14} \\
0 & \sigma_{22} & \sigma_{23} & \sigma_{24} \\
\sigma_{13} & \sigma_{23} & \sigma_{33} & 0 \\
\sigma_{14} & \sigma_{24} & 0 & \sigma_{44} 
\end{pmatrix}
\end{align*}
where $ \sigma_{11} \geq \sigma_{22}>0$ and $\sigma_{33} \geq \sigma_{44}>0$ which exists by Lemma \ref{existssignmat}. Note that $\sigma_{ij}$ is not the $(i,j)$-th entry of $\Sigma^{-1}$ but of $W^t \Sigma^{-1} W$. Let $v_1 = (v_{11},v_{21})$ be the eigenvector of $W_1^t \Sigma^{12}\Sigma^{21} W_1$ associated with the largest eigenvalue. If $\sigma_{11}=\sigma_{22}$ or $\sigma_{33}=\sigma_{44}$, any orthogonal $W_1$ or $W_2$ gives the desired form. In this case, we may chose $W_1$ or $W_2$ such that $ v_{21}\sigma_{24} ( v_{21} \sigma_{24}+ v_{11} \sigma_{14}) \geq 0$, and it is such a choice we fix. Then we have the following theorem.

\begin{theorem}\label{infcor}
The vector $(X_1^2+X_2^2,X_3^2+X_4^2)$ is infinitely divisible if one of the following equivalent conditions is satisfied.
\begin{enumerate}[(i)]
\item There exists a $(2,2)$-signature matrix $U$ such that $U^t \Sigma^{-1} U$ has non-positive off-diagonal elements. 
\item The inequality $ v_{21}\sigma_{24} ( v_{21} \sigma_{24}+ v_{11} \sigma_{14}) \geq 0$ holds. 
\end{enumerate}
\end{theorem}

Example \ref{seconddeltaepsilonexample} builds intuition about condition (ii) above, in particular that the condition holds in cases where $(X_1^2,X_2^2,X_3^2,X_4^2)$ is not infinitely divisible, but also that it is not always satisfied. 

Theorem 2.6 (i) holds for general $n_1,n_2 \geq 1$ as the following result shows. We give the proof below since it is short and makes the need for signature matrices clear. The proof of the more applicable condition (ii) in Theorem \ref{infcor} is postponed to Section 4 since it relies on results that will be establish in that section. 

\begin{corollary}[to Theorem \ref{MarcusAndRosen}]\label{infcor1}
Let $(X_1,\dots,X_{n_1+n_2})$ be a mean zero Gaussian vector with positive definite covariance matrix $\Sigma$. Then 
\begin{align}\label{whatweareinterestedin}
(X_1^2+\dots +X_{n_1}^2,X_{n_1 +1}^2 +\dots +X_{n_1+n_2}^2)
\end{align}
is infinitely divisible if there exists an $(n_1,n_2)$-signature matrix $U$ such that $U^t \Sigma^{-1} U$ has non-positive off-diagonal elements. 
\end{corollary}

\begin{proof}\label{proofofcor}
Write $X= (X_1,\dots, X_{n_1})$ and $Y=(X_{n_1 + 1},\dots, X_{n_1+n_2})$, and note that 
\begin{align}\label{invariance}
\begin{aligned}
 (X_1^2+\dots +X_{n_1}^2,X_{n_1+1}^2+\dots +X_{n_1+n_2}^2) & = (\Vert X\Vert ^2, \Vert Y \Vert^2) \\
 & = (\Vert U_1X\Vert ^2, \Vert U_2Y \Vert^2)
 \end{aligned}
\end{align} 
for any $n_1  \times n_1 $ orthogonal matrix $U_1$ and $n_2 \times n_2$ orthogonal matrix $U_2$. Consequently, any property of the distribution of (\ref{whatweareinterestedin}) is invariant under transformations of the form
\begin{align*}
\begin{pmatrix}
U_1^t & 0 \\
0 & U_2^t
\end{pmatrix} \Sigma \begin{pmatrix}
U_1 & 0 \\
0 & U_2
\end{pmatrix} 
\end{align*}
of the covariance matrix $\Sigma$. Therefore, when there exists an $(n_1,n_2)$-signature matrix $U$ such that $U^t \Sigma^{-1} U$ has non-positive off-diagonal elements, Theorem~ \ref{MarcusAndRosen} ensures infinite divisibility of (\ref{invariance}).  
%Finding a $(2,2)$-signature matrix such that condition i) in Theorem \ref{infcor} is satisfied is in general not obvious, and here condition ii) is more useful as this is readily calculated. 
\end{proof}

\section{Examples and numerics}

We begin this section by presenting two examples treating the inequalities in Theorem \ref{Tracethm} (ii) and Theorem \ref{infcor} (ii) in special cases. Then we calculate the sums in Theorem \ref{Tracethm} numerically with $n_1 =n_2 =2$ for a specific value of $Q$ for $k$ and $m$ less than $60$. 

\begin{example}\label{firstdeltaepsiloneksempel}
Fix $a >0$ and assume that $Q$ is on the form
\begin{align*}
Q = \begin{pmatrix}
Q_{11} & Q_{12} \\
Q_{21} & Q_{22}
\end{pmatrix} = \begin{pmatrix}
q_1 & 0 & \varepsilon   & \varepsilon \\
0 & q_2 & \varepsilon &-\delta \\
\varepsilon & \varepsilon & q_3 & 0 \\
\varepsilon & - \delta  & 0 & q_4 
\end{pmatrix}
\end{align*}
where $\delta, \varepsilon > 0$, $q_1 > q_2>0$, and $q_3 > q_4>0$. Let $v_1= (v_{11},v_{21})$ be the eigenvector of 
\begin{align*}
Q_{12} Q_{21} = \begin{pmatrix}
2 \varepsilon^2 & \varepsilon ( \varepsilon - \delta) \\
\varepsilon ( \varepsilon - \delta) & \varepsilon^2 + \delta^2 
\end{pmatrix}
\end{align*}
associated with the largest eigenvalue $\lambda_1$. We will argue that the inequality in Theorem \ref{Suffcond} (i), which reads  
\begin{align}\label{ineq}
v_{11}  (v_{11} +  v_{21} ) \geq 0
\end{align}
in this case, holds if and only if $\delta \leq \varepsilon$. Then the same theorem will imply that 
\begin{align*}
\trace\: Q_{11}^{k_1} Q_{12}Q_{22}^{m_1} Q_{21} Q_{11}^{k_1} \cdots Q_{11}^{k_{d} }Q_{12}Q_{22}^{m_{d}} Q_{21} Q_{11}^{k_{d+1}} \geq 0
\end{align*} 
for all $d \in \NN_0$ and $k_1,\dots, k_{d+1}m_1,\dots,m_d \in \NN_0$ if and only if $\delta \leq \varepsilon$, and therefore also that the sum in (\ref{sumoftrace}) is non-negative whenever this is the case. 

Since $-v_1$ also is an eigenvector of $Q_{12} Q_{21}$ associated with the largest eigenvalue, we assume $v_{11} \geq 0$ without loss of generality. Assume $\delta \leq \varepsilon$. If $\delta = \varepsilon$, $v_1 = (1,0)$ and the inequality in (\ref{ineq}) holds. Assume $\delta < \varepsilon$. Since $\lambda_1$ is the largest eigenvalue, 
\begin{align*}
\lambda_1 = \sup_{\vert v \vert =1} v^tQ_{12}Q_{21} v  \geq 2 \varepsilon^2
\end{align*}
which implies that 
\begin{align*}
 2 \varepsilon^2 - \lambda_1  \leq 0 \leq \varepsilon ( \varepsilon - \delta).
\end{align*}
Since $v_1$ is an eigenvector, $(Q - \lambda_1) v_1 = 0$ and we therefore have that 
\begin{align*}
0 = ( 2 \varepsilon^2 - \lambda_1) v_{11} +  \varepsilon (  \varepsilon -\delta ) v_{21} \leq \varepsilon ( \varepsilon - \delta) (v_{11} + v_{21} ).
\end{align*}
We conclude that (\ref{ineq}) holds. 

On the other hand, assume $\delta > \varepsilon$ and $v_{11} \geq 0$. Since $\lambda_1$ is the largest eigenvalue, $\lambda_1  \geq \delta^2 + \varepsilon^2 > \delta \varepsilon + \varepsilon^2$ and therefore,
\begin{align*}
(\lambda_1 - 2 \varepsilon^2) > \varepsilon (  \delta -\varepsilon).
\end{align*}
Note that $v_{11}$ can not be zero since the off-diagonal element in $Q_{12}Q_{21}$ is non-zero. We conclude that 
\begin{align*}
0 = ( \lambda_1 - 2 \varepsilon^2)v_{11}   +\varepsilon ( \delta - \varepsilon)  v_{21}  >  \varepsilon ( \delta - \varepsilon)(v_{11} + v_{21}).
\end{align*}
This implies that (\ref{ineq}) does not hold.

\end{example}

\begin{example}\label{seconddeltaepsilonexample}
Assume $\Sigma^{-1}$ is on the form 
\begin{align}\label{deltaepsilonform}
\Sigma^{-1} = \begin{pmatrix}
\Sigma^{11} & \Sigma^{12} \\
\Sigma^{21} & \Sigma^{22}
\end{pmatrix} =\begin{pmatrix}
\sigma_{1} & 0 & -\delta   & \varepsilon \\
0 & \sigma_{2} & \varepsilon &\varepsilon \\
-\delta & \varepsilon & \sigma_{3} & 0 \\
\varepsilon & \varepsilon & 0 & \sigma_{4} 
\end{pmatrix}
\end{align}
where $\sigma_1 >  \sigma_2 >0$, $\sigma_3 > \sigma_4 >0$, and $\delta,\varepsilon > 0$. Let $v_1 = (v_{11},v_{21})$ be the eigenvector of $\Sigma^{12} \Sigma^{21}$ associated with the largest eigenvalue. We will argue that the inequality in Theorem \ref{infcor} (ii) holds if and only if $\delta \leq \varepsilon$. Then the same theorem implies that $(X_1^2+X_2^2,X_3^2+X_4^2)$ is infinitely divisible whenever $\delta \leq \varepsilon$. On the other hand, Theorem \ref{MarcusAndRosen} implies that $(X_1^2,X_2^2,X_3^2,X_4^2)$ is never infinite divisible under (\ref{deltaepsilonform}) since there does not exists a matrix $D$ on the form $\diag(\pm 1,\pm 1, \pm 1, \pm 1)$ such that $D \Sigma^{-1}D$ has non-positive off-diagonal elements. Indeed, for any two matrices $D_1$ and $D_2$ on the form $\diag (\pm 1, \pm 1)$, $D_1 \Sigma^{12} D_2$ has either three negative and one positive or one negative and three positive entrances. 

To see that $v_{21} (v_{11} + v_{21}) \geq 0$ if and only if $\delta \leq \varepsilon$, let 
\begin{align*}
P =\begin{pmatrix}
0 & 1 \\
1 & 0
\end{pmatrix}
\end{align*}
and $Q_{12}$ be given as in Example \ref{firstdeltaepsiloneksempel}. Then $P \Sigma^{12} P = Q_{12}$, implying that $(v_{21},v_{11})$ is the eigenvector associated with the largest eigenvalue of $Q_{12} Q_{21}$. We have argued in Example \ref{firstdeltaepsiloneksempel} that $v_{21}  ( v_{11} + v_{21}) \geq 0$ holds if and only if $\delta \leq \varepsilon$ which is the desired conclusion. 
\end{example}

Now we investigate infinite divisibility of $(X_1^2+X_2^2,X_3^2+X_4^2)$ numerically. More specifically, we consider the sums in (\ref{sumoftrace}) with $n_1=n_2=2$ for a specific choice of positive definite matrix and different values of $k$ and $m$. We will scale $Q$ to have its largest eigenvalue equal to one to avoid getting too close to zero. Due to Theorem \ref{Suffcond} the case where $v_{11} q_{13} ( v_{11} q_{13}+ v_{21} q_{23})<0$ (in the notation from Theorem \ref{Suffcond}) is the only case where the infinite divisibility of $(X_1^2+X_2^2,X_3^2+X_4^2)$ is open.

\normalsize
Let 
\begin{align*}
Q = \frac{1}{\lambda}\begin{pmatrix}
0.8 & 0 & 0.01 & 0.01 \\
0 & 0.3 & 0.01 &-0.2  \\
 0.01& 0.01 & 0.8 & 0 \\
0.01 & -0.2 & 0 & 0.3
\end{pmatrix} 
\end{align*}
where $\lambda>0$ is chosen such that $Q$ has its largest eigenvalue equal to $1$. Note that by Example \ref{firstdeltaepsiloneksempel}, $v_{11} q_{13} ( v_{11} q_{13}+ v_{21} q_{23})<0$. In Figure \ref{fig2} the logarithm of the sums in (\ref{sumoftrace}) for $k$ and $m$ between $0$ and $60$ is plotted. It is seen that the logarithm seems stable and therefore, that the sums in (\ref{sumoftrace}) remain positive in this case. A similar analysis have been done for other positive definite matrices, and we have not encountered any $k,m\in \NN_0$ such that (\ref{sumoftrace}) is negative. This, together with Theorem \ref{Suffcond} (ii), leads us to conjecture that $(X_1^2+X_2^2,X_3^2+X_4^2)$ is infinite divisible for any covariance structure of $(X_1,X_2,X_3,X_4)$. 

\begin{figure}[h]
\centering
\includegraphics[scale=.15]{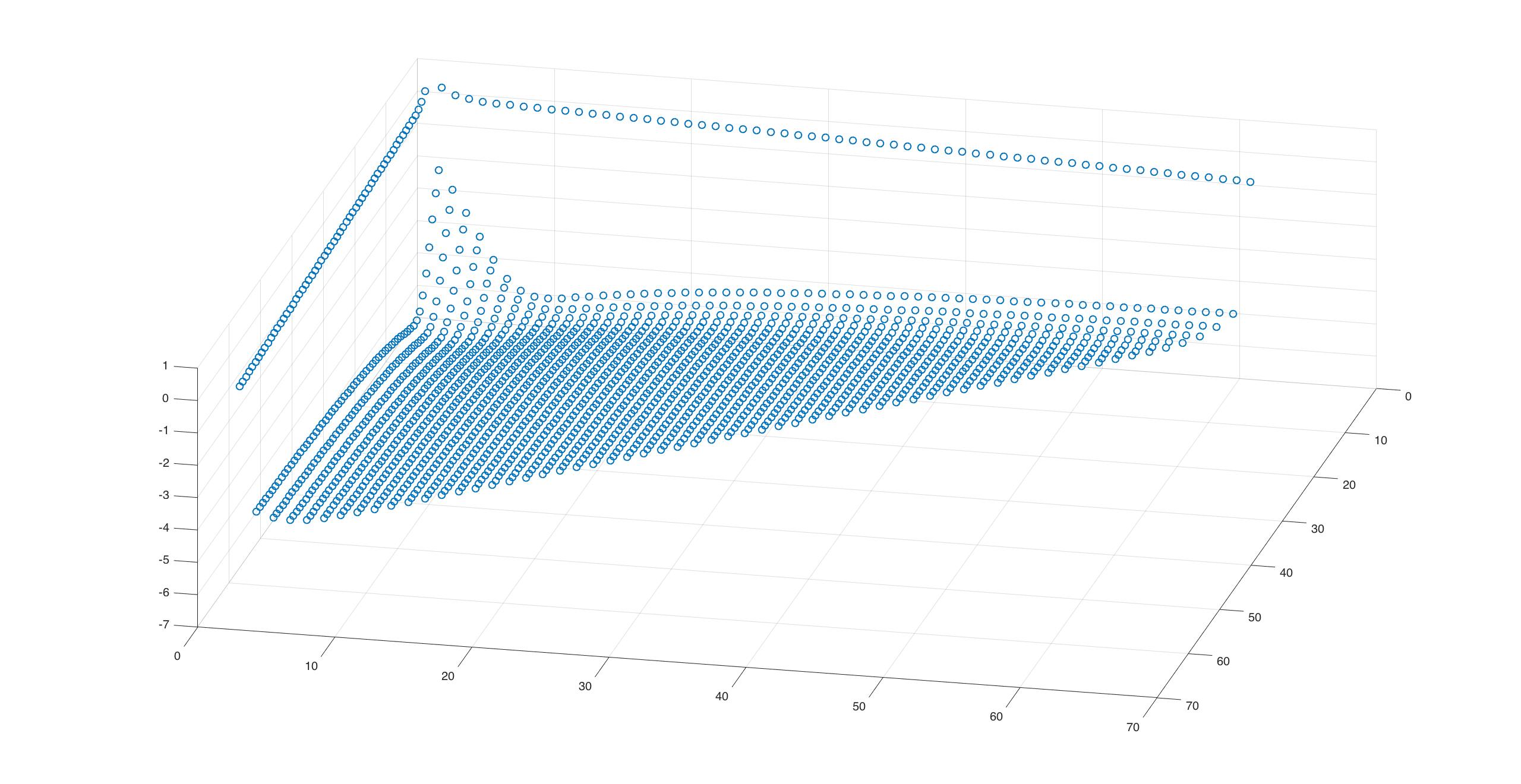}
\caption{The logarithm of the sums in (\ref{sumoftrace}) for $k$ and $m$ between $0$ and $60$.}\label{fig2}
\end{figure}

%Proofs

\section{Proofs}

We start this section with two lemmas on linear algebra. Lemma~\ref{Rotprop} will be very useful in the proofs that make up the rest of this section.

\begin{lemma}\label{existssignmat}
Let $A$ be a $n \times n$ positive definite matrix. Let $n_1,n_2 \in \NN$ be such that $n_1+ n_2=n$ and write 
\begin{align*}
A = \begin{pmatrix}
A_{11} & A_{12} \\
A_{21} & A_{22}
\end{pmatrix}
\end{align*}
where $A_{11}$ is an $n_1 \times n_1$ matrix, $A_{22}$ is an $n_2 \times n_2$ matrix, and $A_{12} = A_{21}^t$ is an $n_1 \times n_2$ matrix. Then there exists an $(n_1,n_2)$-signature matrix $W$ such that $W^t A W$ has the form
\begin{align*}
\begin{pmatrix}
\tilde{A}_{11} & \tilde{A}_{12} \\
\tilde{A}_{21} & \tilde{A}_{22}
\end{pmatrix}
\end{align*}
where $\tilde{A}_{11}= \diag(a_1,\dots,a_{n_1})$ and $\tilde{A}_{22}= \diag(a_{n_1+1},\dots,a_{n_1+n_2})$ with $a_i > 0$ for $i=1,\dots, n_1+n_2$, and where $\tilde{A}_{12}= \tilde{A}_{21}^t$. Furthermore, we may choose $W$ such that $a_1 \geq a_2 \geq \dots \geq a_{n_1}$ and $a_{n_1+1} \geq a_{n_1+2} \geq \dots \geq a_{n_1+n_2}$. 
\end{lemma}

\begin{proof}
Since $A$ is positive definite, $A_{11}$ and $A_{22}$ are positive definite. Consequently, by the spectral theorem (see for example \cite[Corollary 6.4.7]{LEON}), there exists an $n_1 \times n_1$ matrix $W_1$ and an $n_2 \times n_2$ matrix $W_2$, both orthogonal, such that $W_1^t A_{11} W_1$ and $W_2^t A_{22} W_2$ are diagonal with positive diagonal entries. Since permutation matrices are orthogonal matrices, we may assume the diagonal is ordered by size in both $W_1^t A_{11} W_1$ and $W_2^t A_{22} W_2$. Consequently, letting
\begin{align*}
W =\begin{pmatrix}
W_1 & 0 \\
0 & W_2
\end{pmatrix},
\end{align*}
implies that $W^t A W$ has the right form. 
\end{proof}

For a fixed eigenvector $v_i$ we call the system $A v_i = \lambda_i v_i$, the system of eigenequations. The $k$'th equation in this system will be called the $k$'th eigenequation associated with $v_i$. 

Let $A$ be a $4 \times 4$ positive definite matrix, and let $W$ be a $(2,2)$-signature such that 
\begin{align*}
 W^t A W  = \begin{pmatrix}
 W_1^t A_{11} W_1 & W_1^t A_{12} W_2 \\
 W_2^t A_{21} W_1 & W_2^t A_{22} W_2
 \end{pmatrix} = \begin{pmatrix}
 a_{11} & 0 & a_{13} & a_{14} \\
 0 & a_{22} & a_{23} & a_{24} \\
a_{13} & a_{23} & a_{33} & 0 \\
a_{14} & a_{24} & 0 & a_{44} 
\end{pmatrix},
\end{align*}
where $a_{11} \geq a_{22} > 0$ and $a_{33} \geq a_{44} >0$ which exists by Lemma \ref{existssignmat}. Note that $a_{ij}$ is not the $(i,j)$-th entry of $A$ but of $W^t A W$. Let $v_1= (v_{11},v_{21})$ be the eigenvector associated with the largest eigenvalue of $W_1^t A_{12}A_{21}W_1$. If $a_{11}=a_{22}$ or $a_{33}=a_{44}$, any orthogonal $W_1$ or $W_2$ give the desired form. In this case, we may chose $W_1$ or $W_2$ such that $v_{11} a_{13} (v_{11}  a_{13}  + v_{21}  a_{23} ) \geq 0$, and it is such a choice we fix. Then the lemma below will play a central role in the proofs of the previously stated results.

\begin{lemma}\label{Rotprop}
In the notation above, the following are equivalent. 
\begin{enumerate}[(i)]
\item There exists a $(2,2)$-signature matrix $U$ such that $U^t A U $ has all entries non-negative.  
\item For any $d \in \NN$ and $k_1, \dots, k_{d+1}, m_1, \dots m_d \in \NN_0$,
\begin{align*}
\trace A_{11}^{k_1} A_{12} A_{22}^{m_1} A_{21} A_{11}^{k_2} \cdots  A_{11}^{k_d} A_{12} A_{22}^{m_d} A_{21} A_{11}^{k_{d+1}} \geq 0.
\end{align*} 
%\item For sufficiently large $k \in \NN$,
%\begin{align*}
%\trace A_{11}^k A_{12} A_{22}^k A_{21} (A_{12}  A_{21})^k \geq 0.
%\end{align*} 
\item  The inequality $v_{11} a_{13} (v_{11}  a_{13}  + v_{21}  a_{23} ) \geq 0$ holds.
\end{enumerate}
\end{lemma}

\begin{proof}
(i) $\Rightarrow$ (ii).  Let 
\begin{align*}
U = \begin{pmatrix}
U_1 & 0 \\
0 & U_2
\end{pmatrix}
\end{align*}
be such that $B_{ij} = U_i^t A_{ij}U_j$ has non-negative entries for $i,j =1,2$. Then 
\begin{align*}
\MoveEqLeft \trace A_{11}^{k_0} A_{12}A_{22}^{m_1} A_{21} A_{11}^{k_1}\cdots A_{11}^{k_{d-1} }A_{12}A_{22}^{m_{d}} A_{21} A_{11}^{k_d} \\
&=\trace B_{11}^{k_0} B_{12} B_{22}^{m_1} B_{21} B_{11}^{k_1}  \cdots B_{11}^{k_{d-1} } B_{12} B_{22}^{m_{d}} B_{21} B_{11}^{k_d}.
\end{align*}
This trace is non-negative since all matrices in the product only contain non-negative entries.

%(ii) $\Rightarrow$ (iii). This is a special case. 

(ii) $\Rightarrow$ (iii). By the spectral theorem, we may write $W_1^t A_{12}A_{21} W_1= V \Lambda V^t$ where $V$ is a $2 \times 2$ orthogonal matrix and $\Lambda = \diag(\lambda_1,\lambda_2)$ with $\lambda_1 \geq \lambda_2 \geq 0$. Note that $v_1$, the eigenvector associated with largest eigenvalue of $W_1^tA_{12}A_{21}W_1$, is the first column of $V$. If $\lambda_1 = \lambda_2$, $v_1 = (1,0)$ and the inequality holds. If $a_{11}= a_{22}$ or $a_{33}=a_{44}$, $W_1^t A_{11} W_1 = A_{11}$ or $W_2^t A_{22}W_2 = A_{22}$, and choosing $W_1$ or $W_2$ such that $a_{23} = 0$ then ensures the inequality in (iii) holds. 

Assume now that $\lambda_1 > \lambda_2$, $a_{11} > a_{22}$, and $a_{33} > a_{44}$. It follows by assumption that
\begin{align*}
0 &\leq \frac{1}{a_{11}^k} \frac{1}{a_{33}^k} \frac{1}{\lambda_1^k} \trace A_{11}^k A_{12} A_{22}^k A_{21} (A_{12}  A_{21})^k \\
&= \trace \begin{pmatrix}
1 & 0 \\ 0 & (\frac{a_{22}}{a_{11}})^k 
\end{pmatrix} 
W_1^t A_{12} W_2
 \begin{pmatrix}
1 & 0 \\ 0 & (\frac{a_{44}}{a_{33}})^k 
\end{pmatrix}
W_2^t A_{21} W_1
 V \begin{pmatrix}
1 & 0 \\ 0 & (\frac{\lambda_1}{\lambda_2})^k
\end{pmatrix} V^t \\
&  \to \trace \begin{pmatrix}
1 & 0 \\ 0 & 0
\end{pmatrix} W_1^t A_{12} W_2 \begin{pmatrix}
1 & 0 \\ 0 & 0 
\end{pmatrix}W_2^t A_{21}W_1 V \begin{pmatrix}
1 & 0 \\ 0 &0
\end{pmatrix} V^t  
\end{align*} 
as $k \to \infty$. This gives the inequality in (iii) since
\begin{align*}
\MoveEqLeft \trace \begin{pmatrix}
1 & 0 \\ 0 & 0
\end{pmatrix} W_1^t A_{12} W_2 \begin{pmatrix}
1 & 0 \\ 0 & 0 
\end{pmatrix}W_2^tA_{21}W_1 V \begin{pmatrix}
1 & 0 \\ 0 &0
\end{pmatrix} V^t   \\
 & = v_{11} a_{13} (v_{11}  a_{13}  + v_{21} a_{23} ).
\end{align*}

(iii) $\Rightarrow$ (i). To ease the notation and without loss of generality assume that $W = I$. We are then pursuing two $2 \times 2$ orthogonal matrices $U_1$ and $U_2$ such that $U_1^t A_{11} U_1, U_1^t A_{12} U_2$, and $U_2^t A_{22} U_2$ all have non-negative entrances. Initially consider $D_1$ and $D_2$ on the form $\diag(\pm 1, \pm 1)$. Then clearly,  $D_1 A_{11} D_1 = A_{11}$ and $D_2 A_{22} D_2=A_{22}$ since $A_{11}$ and $A_{22}$ are diagonal matrices. Next, note that either it is possible to find $D_1$ and $D_2$  such that $D_1 A_{12} D_2$ has all entrances non-negative or such that 
\begin{align}\label{matform}
D_1 A_{12} D_2 = \begin{pmatrix*}[r]
a_{13} & a_{14} \\
a_{23} & -a_{24}
\end{pmatrix*}
\end{align}
where $a_{13},a_{23},a_{14},a_{24} > 0$. Consequently, we will assume $A_{12}$ is on the form in (\ref{matform}) since otherwise choosing $U_1 = D_1$ and $U_2 = D_2$ would be sufficient. 

As one of two cases, assume $ a_{13}a_{23} - a_{14}a_{24} \geq 0$, and define 
\begin{align*}
U_2 = \begin{pmatrix}
\alpha   \frac{a_{14} a_{24}}{ a_{23}} & \beta a_{23} \\
\alpha a_{14} & -\beta a_{24}
\end{pmatrix}
\end{align*}
where $\alpha, \beta > 0$ are chosen such that each column in $U_2$ has norm one. Then $U_2$ is orthogonal, 
\begin{align*}
A_{12} U_2 = \begin{pmatrix}
\alpha (a_{14}^2 + \frac{a_{13}a_{14}a_{24}}{a_{23}}) & \beta (a_{13}a_{23}-a_{14}a_{24}) \\
0 & \beta (a_{23}^2+a_{24}^2)
\end{pmatrix}, 
\end{align*}
and 
\begin{align*}
U_2^t A_{22} U_2 = \begin{pmatrix}
\alpha^2 \left( a_{33} \left( \frac{a_{14}a_{24}}{a_{23}} \right)^2 + a_{44} a_{14}^2 \right) & \alpha \beta a_{14}a_{24} (a_{33} - a_{44}) \\
\alpha \beta a_{14} a_{24} (a_{33} - a_{44}) & \beta^2 a_{23}^2 + \beta^2 a_{24}^2 
\end{pmatrix}.
\end{align*} 
Since $a_{33} \geq a_{44}$, all entries in $A_{12} U_2 $ and $U_2^t A_{22} U_2 $ are non-negative. Choosing $U_1= I$ then gives a pair of orthogonal matrices with the desired property. 

Now assume $a_{13}a_{23} - a_{14}a_{24} < 0$. Note that $A_{12}$ on the form (\ref{matform}) can not be singular and consequently, there exists $\lambda_1 \geq \lambda_2 > 0$ and an orthogonal matrix $V$ such that $A_{12}A_{21} = V \Lambda V^t$, where $\Lambda = \diag(\lambda_1,\lambda_2)$. Furthermore, since $V$ contains the eigenvectors of $A_{12}A_{21}$ we may assume $v_{11}$ and $v_{12}$ have the same sign where $v_{ij}$ is the $(i,j)$-th component of $V$. Define 
\begin{align}\label{defW}
W = A_{21} V ( \Lambda^{1/2} )^{-1},
\end{align}
and note that this is an orthogonal matrix which, together with $V$, decomposes $A_{12}$ into its singular value decomposition, that is, $V^t A_{12} W = \Lambda^{1/2}$. Then
\begin{align*}
V^t A_{11} V = \begin{pmatrix}
a_{11} v_{11}^2 + a_{22} v_{21}^2 & v_{11} v_{12} (a_{11} - a_{22}) \\
v_{11} v_{12} (a_{11} - a_{22}) & a_{11} v_{12}^2 + a_{22} v_{22}^2 
\end{pmatrix}.
\end{align*}
All entries in $V^t A_{11} V$ are non-negative since we chose $v_{11}$ and $v_{12}$ to have the same sign, and since $a_{11} \geq a_{22}>0$. 

To see that $W^t A_{22} W$ also have all entries non-negative, consider the first line in the eigenequations for $A_{12}A_{21}$ associated with the eigenvector $(v_{12},v_{22})$, the eigenvector associated with the smallest eigenvalue $\lambda_2$,
\begin{align}\label{eigeneqn}
(a_{13}^2+a_{14}^2 - \lambda_2) v_{12} + (a_{13}a_{23} -a_{14}a_{24})v_{22} = 0. 
\end{align}
Since $\lambda_2$ is the smallest eigenvalue of $A_{12}A_{21}$, 
\begin{align*}
 \lambda_2 = \inf_{\vert v \vert =1} v^tA_{12}A_{21} v,
 \end{align*}
 and since the off-diagonal elements in $A_{12}A_{21}$ are non-zero, $(1,0)$ and $(0,1)$ cannot be eigenvectors. Consequently, $\lambda_2$ is strictly smaller than any diagonal element of $A_{12}A_{21}$, and in particular $a_{13}^2+a_{14}^2 - \lambda_2 > 0$. Since we also have $a_{13}a_{23} -a_{14}a_{24} < 0$, (\ref{eigeneqn}) gives that $v_{12}$ and $v_{22}$ need to have the same sign for the sum to equal zero. Let $w_{ij}$ be the $(i,j)$-th component of $W$ and note that by (\ref{defW}), 
\begin{align*}
w_{11}w_{12} = \frac{v_{11} a_{13} + v_{21}a_{23}}{\lambda_1^{1/2}} \frac{v_{12} a_{13} + v_{22}a_{23}}{\lambda_2^{1/2}}.
\end{align*} 
The assumption $v_{11} a_{13} (v_{11} a_{13} + v_{21}a_{23}) \geq 0$ implies that $v_{11} a_{13} + v_{21}a_{23}$ and $v_{11}$ have the same sign. Since $v_{11}$ and $v_{12}$ were chosen to have the same sign, and $v_{12}$ and $v_{22}$ have the same sign, we conclude that $(v_{11} a_{13} + v_{21}a_{23})(v_{12} a_{13} + v_{22}a_{23})$ is non-negative and therefore, $w_{11}w_{12}$ is non-negative too. Then writing 
\begin{align*}
W^t A_{22} W =  \begin{pmatrix}
a_{33} w_{11}^2 + a_{44} w_{21}^2 & w_{11} w_{12} (a_{33} - a_{44}) \\
w_{11} w_{12} (a_{33} - a_{44}) & a_{33} w_{12}^2 + a_{44} w_{22}^2 
\end{pmatrix}
\end{align*}
makes it clear that $W^t A_{22} W$ has non-negative elements. Thus, letting $U_1= V$ and $U_2= W$ completes the proof. 
\end{proof}

%When considering the proof of the following corollary it is apparent that for a $4 \times 4$ matrix $A$ then finding a $(2,2)$-signature matrix $U$ such that $U^t A U$ has non-negative elements or non-positive off-diagonal elements is closely related.

\begin{corollary}\label{Rotpropcor}
Let $A$ and $v_1$ be given as in Lemma \ref{Rotprop}. Then there exists a $(2,2)$-signature matrix $U$ such that $U^t A U$ has non-positive off-diagonal elements if and only if 
\begin{align}\label{ineqsect4}
v_{21} a_{24} ( v_{21} a_{24} + v_{11} a_{14}) \geq 0. 
\end{align}
\end{corollary}

\begin{proof}
Let $W$ be defined as in Lemma \ref{Rotprop}. Define 
\begin{align*}
P_1 = \begin{pmatrix}
0 & 1 \\
1 & 0
\end{pmatrix} \quad \text{and } \quad P = \begin{pmatrix}
P_1 & 0 \\
 0 & P_1
\end{pmatrix}.
\end{align*}
Then $P_1 v_1 = (v_{21},v_{11})$ is the eigenvector of $P_1 W_1^t A_{12} A_{21} W_1 P_1$ associated with the largest eigenvalue. Let  
\begin{align*}
\tilde{A} = \begin{pmatrix}
W_1^t A_{11} W_1 & P_1 W_1^t A_{12}W_2 P_1\\
P_1 W_2^t A_{21} W_1 P_1 & W_2^t A_{22}W_2
\end{pmatrix} =  \begin{pmatrix}
 a_{11} & 0 & a_{24} & a_{23} \\
 0 & a_{22} & a_{14} & a_{13} \\
a_{24} & a_{14} & a_{33} & 0 \\
a_{23} & a_{13} & 0 & a_{44} 
\end{pmatrix}.
\end{align*}
By Lemma \ref{Rotprop}, there exists a $(2,2)$-signature matrix 
\begin{align*}
\tilde{U} = \begin{pmatrix}
\tilde{U}_1 & 0 \\
0 & \tilde{U}_2
\end{pmatrix}
\end{align*} 
such that $\tilde{U}^t \tilde{A} \tilde{U}$ has non-negative entries if and only if $v_{21} a_{24} ( v_{21} a_{24} + v_{11} a_{14}) \geq~0$. Define now the $(2,2)$-signature matrix $U$ as 
\begin{align*}
U = \begin{pmatrix}
U_1 & 0 \\
0 & U_2 
\end{pmatrix} = 
\begin{pmatrix}
- W_1 P_1 \tilde{U}_1 & 0 \\
0 & W_2 P_1 \tilde{U}_2
\end{pmatrix}.
\end{align*}
Let $\tilde{u}_{ij}$ be the $(i,j)$-th component of $\tilde{U}_1$. Since $\tilde{U}_1$ is orthogonal, $\tilde{u}_{12}\tilde{u}_{22} = -\tilde{u}_{11} \tilde{u}_{21}$ implying that 
\begin{align*}
U_1^t A_{11} U_1 = \begin{pmatrix}
\tilde{u}_{11}^2 a_{22} + \tilde{u}_{21}^2a_{11} & \tilde{u}_{11} \tilde{u}_{12}(a_{22}-a_{11}) \\
 \tilde{u}_{11}\tilde{u}_{12}(a_{22}-a_{11}) & \tilde{u}_{12}^2 a_{22} + \tilde{u}_{22}^2a_{11}
\end{pmatrix}
\end{align*}
and 
\begin{align*}
\tilde{U}^t_1 W_1^t  A_{11}  W_1 \tilde{U}_1=  \begin{pmatrix}
\tilde{u}_{11}^2 a_{11} + \tilde{u}_{21}^2a_{22} & \tilde{u}_{11} \tilde{u}_{12}(a_{11}-a_{22}) \\
 \tilde{u}_{11}\tilde{u}_{12}(a_{11}-a_{22}) & \tilde{u}_{12}^2 a_{11} + \tilde{u}_{22}^2a_{22}
\end{pmatrix}.
\end{align*}
Consequently $\tilde{U}^t_1 W_1^t  A_{11}  W_1 \tilde{U}_1$ has non-negative elements if and only if $U_1^t A_{11} U_1$ has non-positive off-diagonal elements. Similarly, $\tilde{U}^t_2 W_2^t  A_{22}  W_2 \tilde{U}_2$ has non-negative elements if and only if $U_2^t A_{22} U_2$ has non-positive off-diagonal elements by a similar argument. Finally we note that 
\begin{align*}
U_1^t A_{12} U_2 = - \tilde{U}^t_1 P_1 W_1^t A_{12} W_2 P_1 \tilde{U}_2,
\end{align*}
and it follows that $U^t A U$ has non-positive off-diagonal elements if and only if 
\begin{align*}
\tilde{U}_1^t P_1 W_1^t A_{12} W_2 P_1 \tilde{U}_2, \quad \tilde{U}_1^t W_1^t A_{11} W_1 \tilde{U}_1 \quad \text{and,} \quad \tilde{U}_2^t W_2 A_{22} W_2 \tilde{U}_2
\end{align*}
have all entries non-negative. We conclude that we can find a $(2,2)$-signature matrix $U$ such that $U^t A U$ has non-positive off-diagonal element if and only if (\ref{ineqsect4}) holds. 
\end{proof}

The following lemma will be useful in the proof of Theorem \ref{Tracethm}. A proof can be found in \cite[Lemma 13.2.2]{MARCUS}.

\begin{lemma}\label{inflemma}
Let $\psi : \RR_+^n \rightarrow (0,\infty)$ be a continuous function. Suppose that, for all $a>0$ sufficiently large, $\log \psi (a(1-s_1),\dots,a(1-s_n))$ has a power series expansion for $s=(s_1,\dots,s_n) \in [0,1]^n$ around $s=0$ with all its coefficients non-negative, except for the constant term. Then $\psi$ is the Laplace transform of an infinitely divisible random variable in $\RR_+^n$.  
\end{lemma}

We now give the proof of Theorem \ref{Tracethm}, where all the main steps follow similar as in \cite[Proof of Theorem 13.2.1]{MARCUS}, but with several modifications to adjust to a different setting. E.g. there is a difference in the $S$ matrix appearing in the proof.

\begin{proof}[Proof of Theorem \ref{Tracethm}]
By \cite[Lemma 5.2.1]{MARCUS},
\begin{align*}
P(s_1,s_2) &= \EE \exp \{ - \tfrac{1}{2} a( (1-s_1) (X_1^2+\dots + X_{n_1}^2) + (1-s_2) (X_{n_1+1}^2+\cdots + X_{n_2}^2))\}  \\
& = \frac{1}{\vert I + \Sigma a (I - S) \vert^{1/2} },
\end{align*}
where $S$ is the $(n_1+n_2)\times (n_1 +n_2)$ diagonal matrix with $s_1$ on the first $n_1$ diagonal entries and $s_2$ on the remaining $n_2$ diagonal entries. Recall that $Q= I -(I+ a\Sigma)^{-1}$. Then 
\begin{align*}
 P(s_1,s_2)^2 &= \vert I + a \Sigma -a \Sigma S \vert ^{-1}  \\
&= \vert (I-Q)^{-1} - ((I-Q)^{-1} - I)S \vert ^{-1} \\
&= \vert I-Q \vert \vert I- QS \vert ^{-1},
\end{align*}
from which it follows that
\begin{align}\label{Lap}\begin{aligned}
2 \log P(s_1,s_2) &= \log \vert I -Q \vert - \log \vert I - QS \vert \\
&= \log \vert I - Q \vert + \sum_{n=1}^\infty \frac{\trace \{ (QS)^n \} }{n},
\end{aligned}
\end{align}
where the last equality follows from \cite[p. 562]{MARCUS}.
Now assume that the vector $(X_1^2+\cdots + X_{n_1}^2,X_{n_1+1}^2+\cdots + X_{n_1+n_2}^2)$ is infinitely divisible, and write 
\begin{align*}
(X_1^2+\cdots + X_{n_1}^2,X_{n_1+1}^2+\cdots + X_{n_1+n_2}^2) \stackrel{d}{=} Y_{1}^n + \dots + Y_n^n
\end{align*}
where $Y_1^n,\dots Y_n^n$ are $2$-dimensional independent identically distributed stochastic vectors. Let $Y_{ij}^n$ be the $j$-th component of $Y_{i}^n$ and note that $Y_{ij}^n \geq 0$ a.s. for all $i,j,n$. Then
\begin{align*}
P(s_1,s_2)^{1/n} = \EE \exp \{ - \tfrac{1}{2} a( (1-s_1) Y_{11}^n + (1-s_2) Y_{12}^n )\}. 
\end{align*}
That $P^{1/n}(s_1,s_2)$ has a power series expansion with all coefficient non-negative follows from writing 
\begin{align*}
 \exp \{ - \tfrac{1}{2}a ( (1-s_j) Y_{1j}^n)\} =   \exp \{ - \tfrac{1}{2} a  Y_{1j}^n)\} ) \sum_{k=0}^\infty \frac{(s_j a Y_{1j}^n)^k	}{2^k k!} .
\end{align*}
We have that
\begin{align}\label{loglim}
\log P(s_1,s_2) = \lim_{n\rightarrow \infty } (n (P^{1/n}(s_1,s_2) -1)).
\end{align}
Note that $(s_1,s_2) \mapsto n (P^{1/n}(s_1,s_2) -1)$ and all its derivatives converge uniformly on $[0,1)\times [0,1)$ by a Weierstrass M-test (see for example \cite[Theorem 7.10]{RUDIN}). Consequently, we may use \cite[Theorem 7.17]{RUDIN} to conclude that 
\begin{align*}
\frac{\partial^{\alpha + \beta}}{\partial s_1^\alpha \partial s_2^\beta} \lim_{n \rightarrow \infty } (n (P^{1/n}(s_1,s_2) -1))= \lim_{n \rightarrow \infty } \frac{\partial^{\alpha + \beta}}{\partial s_1^\alpha \partial s_2^\beta} (n (P^{1/n}(s_1,s_2) -1))
\end{align*}
for any $\alpha,\beta \in \NN_0$. Thus, that all the terms in the power series expansion of $P^{1/n}(s_1,s_2) $ are non-negative implies that all the terms in the power series representation of $\log P(s_1,s_2)$ except the constant term are non-negative by (\ref{loglim}). By (\ref{Lap}) we conclude that any coefficient in front of $s_1^k s_2^m$ in $\trace \{ (QS)^{k+m} \}$ has to be non-negative for all $k,m \in \NN$ and $a>0$. Expanding out the trace then gives that this is equivalent to non-negativity of the sum in (\ref{sumoftrace}) for all $k,m \in \NN_0$.

On the other hand, if the sum in (\ref{sumoftrace}) is non-negative for all $k,m \in \NN_0$ and $a>0$ sufficiently large, (\ref{Lap}) and Lemma \ref{inflemma} imply that 
\begin{align*}
(X_1^2+\cdots + X_{n_1}^2,X_{n_1+1}^2+\cdots + X_{n_1+n_2}^2)
\end{align*}
is infinitely divisible. 
\end{proof}

\begin{proof}[Proof of Theorem \ref{Suffcond}]

Lemma \ref{Rotprop} implies the equivalence in (i). Now we set out to show (ii), i.e., to show that the sum in Theorem \ref{Tracethm} is non-negative for $k,m \in \NN_0$ such that $k\leq 2$, $m\leq 2$, or $k+m\leq 7$ in the special case $n_1 = n_2 = 2$. To this end, consider a $4 \times 4$ positive definite matrix $Q$ and write 
\begin{align*}
Q = \begin{pmatrix}
Q_{11} & Q_{12} \\
Q_{21} & Q_{22}
\end{pmatrix}
\end{align*}
where $Q_{ij}$ is a $2 \times 2$ matrix for $i,j =1,2$. Let $W_1$ and $W_2$ be two $2 \times 2$ orthogonal matrices and define $P_{ij}= W_i Q_{ij} W_j$. Then  
\begin{align}\label{trick}
\begin{aligned}
\MoveEqLeft \trace Q_{11}^{k_1} Q_{12}Q_{22}^{m_1} Q_{21}  \cdots Q_{12}Q_{22}^{m_{d}} Q_{21} Q_{11}^{k_{d+1}}  \\
 = & \trace P_{11}^{k_1} P_{12}P_{22}^{m_1} P_{21}  \cdots P_{12}P_{22}^{m_{d}} P_{21} P_{11}^{k_{d+1}}.
 \end{aligned}
\end{align}
Consequently (see Lemma \ref{existssignmat}), we may assume, without loss of generality, that $Q_{11}$ and $Q_{22}$ are diagonal with the first diagonal element greater than or equal the other and all entries non-negative. 

Either there exists $D_1$ and $D_2$ on the form $\diag (\pm 1, \pm 1)$ such that $D_1 Q_{12} D_2$ has all entries non-negative or such that 
\begin{align*}
D_1 Q_{12} D_2 = \begin{pmatrix}
q_{13} & q_{23} \\
q_{14} & -q_{24}
\end{pmatrix}
\end{align*}
where $q_{13},q_{23},q_{14},q_{24} >0$. If $D_1 Q_{12} D_2$ has all entries non-negative, writing as in (\ref{trick}) with $W_i$ replaced by $D_i$ implies non-negativity of each individual trace. We conclude that we may assume
\begin{align*}
Q = \begin{pmatrix}
\lambda_1 & 0 & q_{13} & q_{14} \\
0 & \lambda_2 & q_{23} & -q_{24} \\
q_{13} & q_{23} & \lambda_3 & 0 \\
q_{14} & -q_{24} & 0 & \lambda_4 
\end{pmatrix},
\end{align*}
where  $\lambda_1 \geq \lambda_2 \geq 0$ and $\lambda_3 \geq \lambda_4 \geq 0$ and  $q_{13},q_{23},q_{14},q_{24} >0$, without loss of generality.

We now write out the traces in (\ref{sumoftrace}) for specific values of $k$ and $m$ and show non-negativity in each case.

\subsubsection*{$k=0$ or $m=0$}

Assume $k=0$ and fix some $m \in \NN$. Then the terms in the sum in Theorem \ref{Tracethm} reduce to $\trace Q_{22}^m$. Since $Q_{22}$ is positive definite, $Q_{22}^m$ is positive definite. Consequently, $\trace Q_{22}^m >0$. Similarly, when $m=0$ and $k \in \NN$, the terms in the sum in Theorem \ref{Tracethm} reduce to $\trace Q_{11}^k$, which again is positive since $Q_{11}$ is positive definite. 

\subsubsection*{$k=1$ or $m=1$}

Assume $k=1$ and fix some $m\in \NN$. Then (\ref{sumoftrace}) reduces to 
\begin{align*}
\trace Q_{12} Q_{22}^m Q_{21} + \sum_{m_1 = 0}^{m-1} \trace Q_{22}^{m_1} Q_{21}Q_{12} Q_{22}^{m-1-m_1},
\end{align*}
which equals
\begin{align*}
(m+1) \trace Q_{12} Q_{22}^m Q_{21}. 
\end{align*}
Since $Q_{12} = Q_{21}^t$ and $Q_{22}$ is positive definite, $Q_{12}Q_{22}^m Q_{21}$ is positive semi-definite. We conclude that $\trace Q_{12} Q_{22}^m Q_{21} \geq 0$.

Assume $m=1$ and fix some $k \in \NN$. Similar to above, (\ref{sumoftrace}) reduces to 
\begin{align*}
\trace Q_{21} Q_{11}^k Q_{12} + \sum_{k_1 = 0}^{k-1} \trace Q_{11}^{k_1} Q_{12}Q_{21} Q_{11}^{k-1-k_1}.
\end{align*}
That this trace is non-negative follows by arguments similar to those above. 
 
\subsubsection*{$k=2$ or $m=2$}

Assume that $k=2$ and let $m \in \NN$. The case $m=1$ is discussed above. Assume $m \geq 2$. Then (\ref{sumoftrace}) reduces to 
\begin{align*}
\MoveEqLeft \trace Q_{11} Q_{12} Q_{22}^{m-1} Q_{21} \\
&+  \sum_{m_1 + m_2 +1 = m}  \trace Q_{22}^{m_1} Q_{21} Q_{11} Q_{12} Q_{22}^{m_2} \\
&+  \sum_{m_1 + m_2 +2 = m} \trace Q_{12} Q_{22}^{m_1} Q_{21} Q_{12} Q_{22}^{m_2} Q_{21}  \\
&+\sum_{m_1 + m_2 + m_3 +2 = m} \trace Q_{22}^{m_1} Q_{21} Q_{12} Q_{22}^{m_2} Q_{21} Q_{12} Q_{22}^{m_3}. 
\end{align*}
All the traces above are non-negative. To see this, consider for example 
\begin{align*}
\trace Q_{22}^{m_1} Q_{21} Q_{12} Q_{22}^{m_2} Q_{21} Q_{12} Q_{22}^{m_3}
\end{align*}
for some $m_1,m_2,m_3 \in \NN_0$. Since $Q_{22}$ is positive definite it has a unique positive definite square root $Q_{22}^{1/2}$. We conclude that 
\begin{align}\label{klig2}
\MoveEqLeft \trace Q_{22}^{m_1} Q_{21} Q_{12} Q_{22}^{m_2} Q_{21} Q_{12} Q_{22}^{m_3}\notag  \\
& = \trace Q_{22}^{(m_1+m_3)/2} Q_{21} Q_{12} Q_{22}^{m_2} Q_{21} Q_{12} Q_{22}^{(m_1+m_3)/2}.
\end{align}
Note that 
\begin{align*}
Q_{22}^{(m_1+m_3)/2} Q_{21} Q_{12}  = (Q_{21} Q_{12} Q_{22}^{(m_1+m_3)/2})^t, 
\end{align*}
which implies that (\ref{klig2}) is the trace of positive semi-definite matrix and therefore non-negative.

Non-negativity of the traces when $m=2$ and $k\in \NN$ follows by symmetry. 

\subsubsection*{$k=3$ and $m=3$}

In the following we will need to expand traces, and we therefore note that
\begin{align}\label{Verysimpletraceexpansion}
\trace \ Q_{11}^k Q_{12} Q_{22}^m Q_{21} = \lambda_1^k \lambda_3^m q_{13}^2 + \lambda_1^k \lambda_4^m q_{14}^2 + \lambda_2^k \lambda_3^m q_{23}^2 + \lambda_2^k \lambda_4^m q_{24}^2
\end{align}
for any $k,m \in \NN$, and
\begin{align}\label{Simpletraceexpansion}
\begin{aligned}
\MoveEqLeft \trace \ Q_{11}^{k_1} Q_{12} Q_{22}^{m_1} Q_{21} Q_{11}^{k_2} Q_{12} Q_{22}^{m_2} Q_{21}\\
&=  \lambda_1^{k_1+k_2}\lambda_3^{m_1+m_2} q_{13}^4+ \lambda_1^{k_1+k_2}\lambda_4^{m_1+m_2} q_{14}^4 \\
&+ \lambda_2^{k_1+k_2}\lambda_3^{m_1+m_2} q_{23}^4+ \lambda_2^{k_1+k_2}\lambda_4^{m_1+m_2} q_{24}^4 \\
&+\lambda_1^{k_1+ k_2}( \lambda_3^{m_1} \lambda_4^{m_2}+ \lambda_4^{m_1} \lambda_3^{m_2}) q_{13}^2q_{14}^2\\
&+\lambda_2^{k_1+ k_2}( \lambda_3^{m_1} \lambda_4^{m_2} + \lambda_3^{m_2} \lambda_4^{m_1}) q_{23}^2q_{24}^2\\
&+\lambda_3^{m_1 + m_2}(\lambda_1^{ k_1} \lambda_2^{k_2} +\lambda_1^{ k_2} \lambda_2^{k_1} ) q_{13}^2q_{23}^2\\
&+ \lambda_4^{m_1+m_2}  (\lambda_1^{k_1} \lambda_2^{k_2}+\lambda_1^{k_2} \lambda_2^{k_1}) q_{14}^2q_{24}^2\\ 
&- (\lambda_1^{k_1} \lambda_2^{k_2} + \lambda_1^{k_2} \lambda_2^{k_1})(\lambda_3^{m_1} \lambda_4^{m_2} + \lambda_3^{m_2} \lambda_4^{m_1}) q_{13}q_{23}q_{14}q_{24}
\end{aligned}
\end{align}
for any $k_1,k_2,m_1,m_2 \in \NN$.

Assume now $k=3$ and $m=3$ and consider the sum in Theorem \ref{Tracethm}. The sum contains all terms on the form 
\begin{align*}
\trace Q_{11}^{k_1} Q_{12} Q_{22}^2 Q_{21} Q_{11}^{k_2} 
\end{align*}
where $k_1+k_2 =2$ and 
\begin{align*}
\trace Q_{22}^{m_1} Q_{21} Q_{11}^2 Q_{12} Q_{22}^{m_2} 
\end{align*}
where $m_1 + m_2 =2$. All these traces equal
\begin{align*}
\trace Q_{11}^2 Q_{12} Q_{22}^2Q_{21}, 
\end{align*}
and there are all together 6 of these terms. Next, the sum in Theorem \ref{Tracethm} also contains all terms on the form
\begin{align*}
\trace Q_{11}^{k_1} Q_{12} Q_{22}^{m_1} Q_{21} Q_{11}^{k_2} Q_{12} Q_{22}^{m_2} Q_{21} Q_{11}^{k_3}
\end{align*}
where $k_1 + k_2 + k_3 =1$ and $m_1 + m_2 = 1$, and 
\begin{align*}
\trace Q_{22}^{m_1} Q_{21} Q_{11}^{k_1} Q_{12} Q_{22}^{m_2} Q_{21} Q_{11}^{k_2} Q_{12} Q_{22}^{m_3}
\end{align*}
where $m_1 + m_2 + m_3 =1$ and $k_1 + k_2 = 1$. Using both that $\trace AB = \trace BA$ and $\trace A^t = \trace A$ for any two square matrices $A$ and $B$ of the same dimensions we get that all these traces share the common trace
\begin{align*}
\trace Q_{11} Q_{12}Q_{21} Q_{12}Q_{22}Q_{21}.
\end{align*}
All together there are 12 of these terms. Finally, the sum in Theorem \ref{Tracethm} contains the two terms 
\begin{align*}
\trace (Q_{12} Q_{21})^3 \quad \text{and} \quad \trace (Q_{21} Q_{12})^3,
\end{align*} 
which share a common trace. We conclude that the sum in Theorem \ref{Tracethm} reads
\begin{align}\label{sumsect4}
\trace \ \{  6 Q_{11}^2 Q_{12}  Q_{22}^2 Q_{21} +12 Q_{11} Q_{12} Q_{21} Q_{12} Q_{22} Q_{21} + 2 (Q_{12} Q_{21})^3 \}. 
\end{align}
Since $Q_{12} = Q_{21}^t$, $Q_{12}Q_{21}$ is positive semi-definite and consequently, $\trace (Q_{12} Q_{21})^3\geq~0$. Furthermore, we have 
\begin{align*}
\trace \  Q_{11}^2 Q_{12}  Q_{22}^2 Q_{21} = \trace \  Q_{11} Q_{12}  Q_{22}^2 Q_{21} Q_{11} \geq 0. 
\end{align*}
Contrarily, there exists a positive definite matrix $Q$ such that 
\begin{align*}
\trace Q_{11} Q_{12} Q_{21} Q_{12} Q_{22} Q_{21} < 0.
\end{align*}
(To see this, consider $Q$ on the form in Example \ref{firstdeltaepsiloneksempel} with $\varepsilon$ small and $\delta$ large relative to $\varepsilon$.) We will now argue that despite this,(\ref{sumsect4}) remains non-negative. Initially we note that 
\begin{align*}
Q_{11}^{k_i} Q_{12} Q_{22}^{m_i} Q_{21} = \begin{pmatrix}
\lambda_1^{k_i} (\lambda_3^{m_i} q_{13}^2 + \lambda_4^{m_i} q_{14}^2)  & \lambda_1^{k_i} (\lambda_3^{m_i} q_{13}q_{23} -  \lambda_4^{m_i} q_{14}q_{24})  \\
\lambda_2^{k_i}( \lambda_3^{m_i} q_{13}q_{23} - \lambda_4^{m_i} q_{14}q_{24})  & \lambda_2^{k_i} (\lambda_3^{m_i} q_{23}^2 +\lambda_4^{m_i} q_{24}^2 )
\end{pmatrix}
\end{align*}
and 
\begin{align*}
Q_{22}^{m_i} Q_{21} Q_{11}^{k_i} Q_{12} = \begin{pmatrix}
\lambda_3^{m_i} (\lambda_1^{k_i} q_{13}^2 + \lambda_2^{k_i} q_{23}^2)  & \lambda_3^{m_i} (\lambda_1^{k_i} q_{13}q_{14} -  \lambda_2^{k_i} q_{23}q_{24})  \\
\lambda_4^{m_i}( \lambda_1^{k_i} q_{13}q_{14} - \lambda_2^{k_i} q_{23}q_{24})  & \lambda_4^{m_i} (\lambda_1^{k_i} q_{14}^2 +\lambda_2^{k_i} q_{24}^2 )
\end{pmatrix}.
\end{align*}
Since $\lambda_1 \geq \lambda_2$ and $\lambda_3 \geq \lambda_4$, we see that if $q_{13}q_{14} \geq q_{23}q_{24}$ or $q_{13} q_{23} \geq q_{14}q_{24}$, then one of two matrices above have only non-negative entrances for any $k_i, m_i \in \NN_0$. Consequently,
\begin{align*}
\trace Q_{11}^{k_1} Q_{12}Q_{22}^{m_1} Q_{21}Q_{11}^{k_2} Q_{12} Q_{22}^{m_2} Q_{21}  =\trace   Q_{22}^{m_1} Q_{21}Q_{11}^{k_1} Q_{12} Q_{22}^{m_2} Q_{21} Q_{11}^{k_2} Q_{12} 
\end{align*}
would be non-negative if this was the case. Especially, we would have
\begin{align*}
\trace Q_{11} Q_{12} Q_{21} Q_{12} Q_{22} Q_{21} \geq 0.
\end{align*}
Assume now that $q_{13} q_{14} \leq q_{23} q_{24}$ and $q_{13} q_{23} \leq q_{14}q_{24}$. By (\ref{Verysimpletraceexpansion}) and (\ref{Simpletraceexpansion}), 
\begin{align}\label{Ettrace}\begin{aligned}
\MoveEqLeft \trace \ \{ \tfrac{1}{2} Q_{11}^2Q_{12} Q_{22}^2 Q_{21} + Q_{11} Q_{12}Q_{22}Q_{21}Q_{12}Q_{21} \} \\
&=   \tfrac{1}{2} \lambda_1^2 \lambda_3^2 q_{13}^2 +  \tfrac{1}{2} \lambda_1^2 \lambda_4^2 q_{14}^2 +  \tfrac{1}{2} \lambda_2^2 \lambda_3^2 q_{23}^2 +  \tfrac{1}{2} \lambda_2^2 \lambda_4^2 q_{24}^2 \\
&+ \lambda_1\lambda_3 q_{13}^4+ \lambda_1\lambda_4 q_{14}^4 + \lambda_2\lambda_3 q_{23}^4+ \lambda_2\lambda_4 q_{24}^4 \\
& +\lambda_1( \lambda_3 + \lambda_4) q_{13}^2q_{14}^2 +\lambda_2( \lambda_3 +  \lambda_4) q_{23}^2q_{24}^2 \\
& +\lambda_3(\lambda_1  + \lambda_2 ) q_{13}^2q_{23}^2 + \lambda_4 (\lambda_1+ \lambda_2) q_{14}^2q_{24}^2 \\ 
& -  (\lambda_1 + \lambda_2)(\lambda_3 +  \lambda_4) q_{13}q_{23}q_{14}q_{24}.
\end{aligned}
\end{align}
We are going to bound the term $ (\lambda_1 + \lambda_2)(\lambda_3 +  \lambda_4) q_{13}q_{23}q_{14}q_{24}$ by the  positive terms to show non-negative of this trace. We recall that $\lambda_1 \geq \lambda_2>0$ and $\lambda_3 \geq \lambda_4>0$. Initially, note that 
\begin{align*}
\lambda_2 \lambda _3 q_{13}q_{23}q_{14}q_{24}& \leq \lambda_2 \lambda_3 q_{13}^2q_{23}^2 \\
\lambda_2 \lambda _4 q_{13}q_{23}q_{14}q_{24} &\leq \lambda_1 \lambda_4 q_{13}^2q_{14}^2 \\
\lambda_1 \lambda _4 q_{13}q_{23}q_{14}q_{24} &\leq \lambda_1 \lambda_3 q_{13}^2q_{14}^2. 
\end{align*}
This leaves only $\lambda_1 \lambda _3 q_{13}q_{23}q_{14}q_{24}$ to be bounded. If $\lambda_1 \lambda _3 q_{13}q_{23}q_{14}q_{24} \leq \frac{1}{2} \lambda_1^2\lambda_3^2 q_{13}^2$, we have a bounding term in (\ref{Ettrace}). Therefore, assume $2q_{23}q_{14}q_{24} \geq \lambda_1\lambda_3 q_{13}$. Since $Q$ was assumed positive definite, $\lambda_2 \lambda_4 \geq q_{24}^2$. Consequently,
\begin{align*}
\lambda_1 \lambda _3 q_{13}q_{23}q_{14}q_{24} &\leq 2 q_{23}^2q_{14}^2q_{24}^2 \\
&\leq 2 \lambda_2 \lambda_4 q_{23}^2 q_{13}^2 \\
& \leq \lambda_2 \lambda_4 ( q_{23}^4 + q_{13}^4) \\
& \leq \lambda_2 \lambda_3 q_{23}^4 + \lambda_1 \lambda_3 q_{13}^4.
\end{align*}
We conclude that (\ref{Ettrace}) and hence (\ref{sumsect4}) is non-negative.

\subsubsection*{$k+m=7$}

Now consider $k,m \in \NN$ such that $k+m=7$. Whenever $k,m=1,2$, we already know that the sum in Theorem \ref{Tracethm} is non-negative. Let $k=3$ and $m=4$. Then the sum in Theorem \ref{Tracethm} reads 
\begin{align}\label{kplusmequalseven}
\begin{aligned}
\MoveEqLeft \trace \ \{ 14Q_{11} Q_{12} Q_{21} Q_{12} Q_{22}^2 Q_{21} +  7 Q_{11}^2 Q_{12}  Q_{22}^3 Q_{21}\\
& 7 Q_{11} (Q_{12} Q_{22} Q_{21})^2 +7 Q_{12} Q_{22} Q_{21} (Q_{12} Q_{21})^2 \}. 
\end{aligned}
\end{align}
Initially we note that 
\begin{align*}
\trace  Q_{11} (Q_{12} Q_{22} Q_{21})^2  \geq 0 \quad \text{and} \quad  \trace Q_{12} Q_{22} Q_{21} (Q_{12} Q_{21})^2 \geq 0
\end{align*}  
since they both can be written as the trace of positive semi-definite matrices (see above for more details). Next, by (\ref{Verysimpletraceexpansion}) and (\ref{Simpletraceexpansion}),
\begin{align}\label{s13s24}
\begin{aligned}
\MoveEqLeft \trace \ \{ \tfrac{1}{2} Q_{11}^2 Q_{12} Q_{22}^3 Q_{21} + Q_{11} Q_{12}Q_{22}^2Q_{21}Q_{12}Q_{21} \}\\
& = \tfrac{1}{2} \lambda_1^2 \lambda_3^3 q_{13}^2 +  \tfrac{1}{2} \lambda_1^2 \lambda_4^3 q_{14}^2 +  \tfrac{1}{2} \lambda_2^2 \lambda_3^3 q_{23}^2 +  \tfrac{1}{2} \lambda_2^2 \lambda_4^3 q_{24}^2 \\ 
& +\lambda_1( \lambda_3^2 + \lambda_4^2) q_{13}^2q_{14}^2 +\lambda_2( \lambda_3^2 +  \lambda_4^2) q_{23}^2q_{24}^2 \\
& +\lambda_3^2(\lambda_1  + \lambda_2 ) q_{13}^2q_{23}^2 + \lambda_4^2 (\lambda_1+ \lambda_2) q_{14}^2q_{24}^2 \\ 
& +\lambda_1\lambda_3^2 q_{13}^4+ \lambda_1\lambda_4^2 q_{14}^4 + \lambda_2\lambda_3^2 q_{23}^4+ \lambda_2\lambda_4^2 q_{24}^4 \\
& -  (\lambda_1 + \lambda_2)(\lambda_3^2 +  \lambda_4^2) q_{13}q_{23}q_{14}q_{24}.
\end{aligned}
\end{align}
Again we bound the negative term by positive terms. Recall that  $\lambda_1 \geq \lambda_2$ and $\lambda_3 \geq  \lambda_4$, and that we may assume $q_{23}q_{24} \geq q_{13}q_{14}$ and  $q_{14}q_{24} \geq q_{13}q_{23}$ without loss of generality. Consequently, 
\begin{align*}
 \lambda_1 \lambda_4^2 q_{13}q_{23}q_{14}q_{24} &\leq  \lambda_1 \lambda_4^2  q_{14}^2 q_{24}^2\\
 \lambda_2 \lambda_3^2 q_{13}q_{23}q_{14}q_{24} &\leq \lambda_2 \lambda_3^2  q_{23}^2q_{24}^2 \\
 \lambda_2 \lambda_4^2 q_{13}q_{23}q_{14}q_{24} &\leq  \lambda_2 \lambda_4^2  q_{14}^2q_{24}^2,
\end{align*}
leaving  $ \lambda_1 \lambda_3^2 q_{13}q_{23}q_{14}q_{24}$ to be bounded. First note that
\begin{align*}
 \tfrac{1}{2} \lambda_1^2 \lambda_3^3 q_{13}^2- \lambda_1 \lambda_3^2 q_{13}q_{23}q_{14}q_{24} =  \lambda_1 \lambda_3^2 q_{13} ( \tfrac{1}{2} \lambda_1 \lambda_3 q_{13}  - q_{23}q_{14}q_{24}),
\end{align*}
so that non-negativity holds if $  \tfrac{1}{2} \lambda_1 \lambda_3 q_{13}  \geq q_{23}q_{14}q_{24}$. Assume $  \lambda_1 \lambda_3 q_{13}  \leq 2q_{23}q_{14}q_{24}$ and recall that $\lambda_2 \lambda_4 \geq q_{24}^2$ since $Q$ is positive definite. Then 
\begin{align*}
\lambda_1 \lambda_3^2 q_{13}q_{23}q_{14}q_{24} &\leq  2\lambda_3 q_{23}^2 q_{14}^2 q_{24}^2 \\
& \leq 2 \lambda_2 \lambda_3 \lambda_4 q_{23}^2 q_{14}^2 \\
& \leq \lambda_2 \lambda_3^2 q_{23}^4 + \lambda_2 \lambda_4^2 q_{14}^4 \\
&\leq  \lambda_2 \lambda_3^2 q_{23}^4 + \lambda_1 \lambda_4^2 q_{14}^4
\end{align*} 
so we have found bounding terms for the last expression. We conclude that (\ref{s13s24}) is non-negative and therefore, (\ref{kplusmequalseven}) is non-negative too. The case $k=4$ and $m=3$ follows by symmetry. It follows that the sum in Theorem \ref{Tracethm} is non-negative for $k+m=7$. 
\end{proof}

%Proof of sufficient condiotion Corollary

\begin{proof}[Proof of Theorem \ref{infcor}]
Corollary \ref{Rotpropcor} gives that (i) and (ii) are equivalent and Corollary~\ref{infcor1} gives that (i) implies infinite divisibility. 
\end{proof}

% Proof of trace Theorem

\textbf{Acknowledgments.} This research was supported by the Danish Council for Independent Research (Grant DFF - 4002-00003).

\newpage
\bibliographystyle{abbrv}

%\section*{References}

 \bibliography{bibliografi}

%\begin{thebibliography}{}
%
%\bibitem{GR} Griffiths, R. C. (1984): Characterization of infinitely divisible multivariate gamma distributions,
% \textrm{\
%\emph{Journal of Multivariate Analysis}}.
%
%\bibitem{SH} Shanbhag, D. N. (1976), On the structure of the Wishart distribution, {\em Journal of Multivariate Analysis}.
%
%\bibitem{BA} Bapat, R. B. (1989): Infinite divisibility of multivariate gamma distributions and M-matrices,
% \textrm{\
%\emph{Sankhy\textoverline{a}: The Indian Journal of Statistics}}.
%
%\bibitem{MR} Markus, M.B., Rosen, J (2006): Markov Processes, Gaussian Processes, and Local Times,
% \textrm{\
%\emph{Cambridge University Press}}.
%
%\bibitem{SL} Leon, S. J. (2006): Linear Algebra with applications, eight edition,
% \textrm{\
%\emph{Pearson}}.
%
%\bibitem{JA} Janson, S. (1997): Gaussian Hilbert Spaces, 
% \textrm{\
%\emph{Cambridge University Press}}.
%
%\bibitem{LE} Lévy, P.: The arithmatical character of the Wishart distribution. \emph{Proc. Camb. Phil. Soc.
%44, 295–297 (1948)}
%
%
%\end{thebibliography}

\end{document}